\documentclass[12pt]{amsart}

\usepackage{amsfonts}
\usepackage{amscd}
\usepackage{amssymb}
\usepackage{graphics}
\usepackage{amsmath}
\usepackage[english]{babel}
\usepackage{hyperref}

\hypersetup{colorlinks,citecolor=blue}

\newcommand{\BZ}{{\mathbb{Z}}}
\newcommand{\BN}{{\mathbb{N}}}
\newcommand{\BR}{{\mathbb{R}}}
\newcommand{\BC}{{\mathbb{C}}}
\newcommand{\BF}{{\mathbb{F}}}

\newcommand{\gD}{\Delta}
\newcommand{\gd}{\delta}

\newcommand{\gC}{\Gamma}
\newcommand{\gc}{\gamma}
\newcommand{\gs}{\sigma}
\newcommand{\gS}{\Sigma}

\newcommand{\gep}{\epsilon}

\newcommand{\gL}{\Lambda}
\newcommand{\ga}{\alpha}
\newcommand{\gt}{\tau}

\newcommand{\ti}[1]{\tilde{#1}}

\newcommand{\rank}{\mathrm{rank}}

\newcommand{\vol}{\mathrm{vol}}

\newcommand{\Min}{\mathrm{Min}}

\newcommand{\SL}{\mathrm{SL}}
\newcommand{\GL}{\mathrm{GL}}
\newcommand{\PSL}{\mathrm{PSL}}

\newcommand{\SO}{\mathrm{SO}}

\newtheorem{prop}{Proposition}[section]
\newtheorem{thm}[prop]{Theorem}
\newtheorem{lem}[prop]{Lemma}
\newtheorem{cor}[prop]{Corollary}

\theoremstyle{definition}

\newtheorem{rem}[prop]{Remark}

\newtheorem{clm}[prop]{Claim}

\title{Volume vs. rank of lattices}
\author{T. Gelander}

\begin{document}

\maketitle

\begin{abstract}
We show the rank (i.e. minimal size of a generating set) of lattices cannot grow faster than the volume. 
\footnote{Classification: 22E40 Discrete subgroups of Lie groups.
Keywords: Lie groups. Lattices.} 
\end{abstract}

\section{The main result and some consequences}

For a group $\gC$ we denote by $d(\gC)$ the minimal cardinality of a generating set. When considering the class of finite index subgroups $\gC$ in some given finitely generated group $\gD$, it is easy to show that $d(\gC)$ is at most $(d(\gD)-1)[\gD:\gC]+1$ and in particular bounded linearly by the index. We prove the analog statement when $\gD$ is replaced by a connected semisimple Lie group $G$.

\begin{thm}\label{thm:d(Gamma)}
Let $G$ be a connected semisimple Lie group without compact factors. Then there is an eﬀectively 
computable constant $C=C(G)$ such that
$$
 d(\gC)\le C\vol(G/\gC),
$$
for every irreducible lattice $\gC\le G$.
\end{thm}

Theorem \ref{thm:d(Gamma)} for rank one groups was proved in \cite[Section 2]{BGLS}. For torsion free lattices Theorem \ref{thm:d(Gamma)} was proved in \cite{Ge} where also bounds an the number of {\it short} relations were given. The proof of the general case encounters several difficulties which do not appear in rank one and in the torsion free case.

Theorem \ref{thm:d(Gamma)} implies in particular the well known but nontrivial fact that every lattice is finitely generated.
Originally the finiteness of $d(\gC)$ was proved case by case by many different authors, notably is the work of Garland and Raghunathan \cite{Ga-Ra} for $\rank(G)=1$ and of Kazhdan \cite{Kazhdan} for the case where all the factors of $G$ have rank $\ge 2$.\footnote{Note in particular that our proof does not rely on Kazhdan's property (T).} 
In \cite{BGS} a geometric proof was given under no assumption on $G$ but assuming that $\gC$ is torsion free, i.e. that $\gC\backslash G/K$ is a manifolds. Furthermore, in \cite{BGS} an analog result was proved for more general manifolds of nonpositive curvature which are not necessarily locally symmetric. Our proof is inspired by \cite{BGS}, but
holds under no assumption on $G$ or on $\gC$. 

Note also that Theorem \ref{thm:d(Gamma)} implies the classical Kazhdan--Margulis theorem:

\begin{cor}[Kazhdan--Margulis Theorem \cite{KM}]
For any semisimple Lie group without compact factors $G$, there is a positive lower bound on the covolume of lattices.  
\end{cor}

\begin{proof}
Indeed, 
$$
 \vol (G/\gC)\ge\frac{d(\gC)}{C}\ge\frac{2}{C},
$$
for every lattice $\gC\le G$.
\end{proof}
Our proof also gives the stronger statement of Kazhdan--Margulis' theorem that there is an identity neighborhood in $G$ which intersects trivially a conjugate of every lattice (see Remark \ref{rem:K-M}).

In general, the linear estimate in Theorem \ref{thm:d(Gamma)} cannot be improved. For instance, for every $n\ge 2$ the group $G=\SO(n,1)$ admits a lattice $\gC$ which projects onto the free group $F_2$ (see \cite{Lub1}), hence taking $\gC_k$ in $\gC$ to be the pre-image of an index $k$ subgroup of $F_2$, we get $d(\gC_k)\ge k+1$ while $\vol(G/\gC_k)=k\cdot\vol (G/\gC)$. On the other hand, in general one cannot give a lower bound on $d(\gC)$ which tends to $\infty$ with $\vol(G/\gC)$. Moreover, by \cite{SV}, when $\rank_\BR(G)\ge 2$, every nonuniform lattice in $G$ admits a $3$-generated finite index subgroup. 

Another immediate consequence of Theorem \ref{thm:d(Gamma)} is that the first Betti number grows at most linearly with the covolume:

\begin{cor}\label{cor:betti}
The first Betti number $b_1(\gC,A)$ of a lattice $\gC\le G$, with respect to an arbitrary ring $A$, is at most $C\vol (G/\gC)$. 
\end{cor}

\begin{proof}
$$
 b_1(\gC,A)\le d(\gC)\le C\vol (G/\gC).
$$
\end{proof}

Corollary \ref{cor:betti} extends Theorem 2 from \cite{BGS} to locally symmetric orbifolds rather than manifolds. However, while in \cite{BGS} it is shown that {\it all} the Betti numbers are bounded by the volume of the manifold, here we obtain a bound only on the first Betti number. It is likely that the higher Betti numbers of orbifolds are bounded by the volume as well. We refer to \cite{Fi-Gr} for some lower bounds on $b_1$ of some arithmetic lattices in $\SL_2(\BC)$. When $G$ has property $(T)$, the first Betti number with coefficients in $\BZ$, $b_1(\gC,\BZ)$ vanishes for every lattice $\gC\le G$, however for rings with torsion $b_1(\gC,A)$ is typically nonzero. For instence, by \cite{Lub2}, every finitely generated linear group admits a finite index subgroup which is mapped onto $\BF_p$ where $p$ is an arbitrary rational prime.

As another consequence we obtain
a concrete upper bound on the number $AL_G(v)$ of conjugacy classes of arithmetic subgroups of $G$ of covolume at most $v$:

\begin{cor}\label{cor:AL}
Let $G$ be a connected semisimple lie group without compact factors. Then for every $\gep>0$
$$
 AL_G(v)\le v^{(1+\gep)Cv},
$$
for all $v\gg 0$.
\end{cor}

Corollary \ref{cor:AL} was proved in \cite{BGLS} for rank one groups (see \cite[Theorem 1.1]{BGLS}). 
Using Theorem \ref{thm:d(Gamma)} instead of the rank one case given in \cite{BGLS}, Corollary \ref{cor:AL} is proved by the argument given in \cite{BGLS} also for higher rank groups.
It was also shown in \cite{BGLS} that for $\SO(n,1)$ this estimate is sharp in the sense that $AL_{\SO(n,1)}(v)\ge v^{av}$ for some constant $a>0$ and $v$ sufficiently large, and for $G=\SL(2,\BR)$ the following precise asymptotic formula was given 
$$
 \log AL_{\SL(2,\BR)}(v)\sim \frac{v}{2\pi}\log v
$$ 
where the Haar measure of $\SL(2,\BR)$ is the one corresponding to the hyperbolic measure on $\SL(2,\BR)/\SO(2)$. When $\rank(G)>1$ the estimate given in Corollary \ref{cor:AL} is somewhat weak. In fact Belolipetsky and Lubotzky \cite{BL} have shown recently that if all lattices in $G$ posses the Congruence Subgroup Property then 
$$
 v^{a\log v}\le AL_G(v)\le v^{b\log v},
$$
for some $a,b>0$ and all $v\gg 0$. It was conjectured by Serre that the CSP holds for all irreducible lattices in all higher rank semisimple Lie groups, and this conjectured was confirmed in many cases. But the general case, and in particular when $\gC$ is a uniform lattice in $\SL_{n\ge 3}(\BR)$, is unknown.  
Corollary \ref{cor:AL} holds under no assumptions. An upper bound as in Corollary \ref{cor:AL} was proved for {\it torsion free} lattices (which are not necessary arithmetic) in \cite{BGLM,Ge} when $G$ was assumed to be not locally isomorphic to $\SL(2,\BR),\SL(2,\BC),\SL(3,\BR)$ and to $\SL(2,\BR)\times\SL(2,\BR)$.
Note however that even when restricting to torsion free lattices, in the special cases of $G=\SL(3,\BR)$ or $\SL(2,\BR)\times\SL(2,\BR)$ no quantitative estimate was given until today.

\medskip

\medskip

\noindent
{\bf About the proof:} The basic idea is to deform the symmetric space $X=G/K$ to nice connected $\gC$-invariant subset $Y\subset X$ where the displacement of every $\gc\in\gC\setminus\{1\}$ is bounded below by a uniform constant, and deduce that $\gC$ is a quotient of $\pi_1(Y/\gC)$ while the topology of $Y/\gC$ is bounded by its volume since it is uniformly thick.
We use techniques similar to \cite{BGS,Ge}.

\section{The proof}

\subsection{Some notation and background}
Let $G$ be a connected center-free semisimple Lie group, $K\le G$ a maximal compact subgroup of $G$ and $X=G/K$ the associated Riemannian symmetric space.
$X$ equipped with the analytic Riemannian metric coming from the Killing form of $G$ is nonpositively curved, i.e.\ the distance function $d:X\times X\to \BR$ is convex. 

For $g\in G,~x\in X$ we denote by 
$$
 d_g(x)=d(x,g\cdot x)
$$ 
the displacement of $g$ at $x$, and by 
$$
 \Min(g):=\{x\in X:d_g(x)=\inf(d_g)\}
$$ 
the set where $d_g$ attains its infimum. Recall that $g$ is {\it semisimple} iff $\Min(g)\ne\emptyset$ (equivalently, iff $g$ belongs to a torus in $G$), and otherwise $g$ is called {\it parabolic}. 

A semisimple element $g$ is {\it elliptic} if $\inf d_g=0$ in which case $\Min(g)$ is also denoted by 
$$
 \text{Fix}(g)=\{x\in X:g\cdot x=x\},
$$ 
and otherwise it is {\it hyperbolic} and admits an axis on which it translates by a positive amount and the set $\Min(g)$ consists of the union of all its axis. In either case $\Min(g)$ is a complete totally geodesic submanifold. 

$G$ admits a structure of a linear real algebraic group, and an element $g\in G$ is semisimple iff it is diagonalizable, and is {\it unipotent} iff all its eigenvalues are $1$. These properties are independent of the chosen representation of $G$. Moreover, a unipotent element $g\in G$ is parabolic with $\inf d_g=0$ (the converse is not true in general).

For $t\ge 0$ we denote by $\{d_g\le t\}$ the sublevel set $\{x:d_g(x)\le t\}$. Since $X$ has nonpositive sectional curvature, $\{ d_g\le t\}$ is convex.

\subsection{A co-dimension 2 condition}

Write $c~\|~c'$ to express that two geodesic lines $c,c'$ are parallel, and denote 
$$
 \mathfrak{I}(c):=\bigcup \{ c'(\BR):c'~\|~c\}
$$
the union of all traces of geodesic lines in $X$ parallel to $c$. 

\begin{lem}\label{lem:codim}
Let $X$ be an irreducible symmetric space of noncompact type. If for some geodesic line $c$, $\mathfrak{I}(c)$ has codimension one, then $X$ is isometric to the hyperbolic plane $\mathcal{H}^2$.
\end{lem} 

\begin{proof}
Let $p^+,p^-\in \partial X$ be the endpoints of $c$ in the ideal boundary of $X$, and let $P^+,P^-\le G=\text{Isom}(X)^\circ$ be the corresponding parabolic subgroups. 
Observe that for $g\in P^+$, $g\cdot c(0)\in\mathfrak{I}(c)$ iff $g\in P^-$. Indeed,  if $g\in P^-\cap P^+$ then $g\cdot c$ is parallel to $c$, and vice versa: if $g\cdot c(0)\in\mathfrak{I}(c)$, the geodesic line through $g\cdot c(0)$ parallel to $c$ has one end $p^+$, and since this geodesic is determined by $g\cdot c(0)$ and $p^+=g\cdot p^+$ it coincides with $g\cdot c$, and its other end point $p^-$ has to be $g\cdot p^-$.
Since $P^+$ acts transitively on $X$, we deduce that 
$$
 \dim P^+/P^+\cap P^-=\dim X-\dim\mathfrak{I}(c)=1.
$$ 
Furthermore, since $P^+P^-$ is open in $G$, being the big cell in a Bruhat decomposition, the orbit $P^+\cdot p^-$ is open in $G\cdot p^-\sim G/P^-$. It follows that $G/P^-$ is one dimensional, hence, being compact, is homeomorphic to the circle $S^1$. However the only center free simple Lie group that acts nontrivially on $S^1$ is $\PSL_2(\BR)$. Hence $G\cong\PSL_2(\BR)$ and $X\cong\mathcal{H}^2$.
\end{proof}

\begin{cor}\label{cor:codim}
Let $G$ be a connected center-free semisimple Lie group without compact factors and not isometric to $\PSL_2(\BR)$, and let $g\in G$ be an element whose projection to every factor of $G$ is nontrivial. 
Then $\text{codim}_X\Min(g)\ge 2$.
\end{cor}

\begin{proof}
If $G$ has more than one factor, $\Min(g)\cong\prod\Min(g_i)$ where $g_i$ are the components of $g$ so the conclusion is trivial since the co-dimension is at least the number of factors. Assume therefore that $G$ is simple.
Since $G$ is connected, every $g\in G$ preserves the orientation of $X$, hence $\text{Fix}(g)$ is of codimension $\ge 2$ for every elliptic element $g\in G$.
For $g$ hyperbolic, since $\Min(g)\subset\mathfrak{I}(c)$ where $c$ is an axis of $g$,
the result follows from Lemma \ref{lem:codim}
\end{proof}



\subsection{Defining the Morse function}

Recall the classical Margulis lemma (see \cite[Proposition 4.1.16]{Th}):

\begin{lem}\label{lem:margulis}
There are two constant $\gep_G>0$ and $m_G\in\BN$ such that if $\gL<G$ is a discrete group generated by $\{\gc\in\gL:d_\gc(x)\le\gep_G\}$ for some $x\in X$, then $\gL$ contains a normal subgroup of index $\le m_G$ which is contained in a connected nilpotent subgroup of $G$.
\end{lem}


Set
$$
 \gep=\frac{\gep_G}{2}~~~~~\text{and}~~~~~m=m_G!
$$
We will also use the following:

\begin{lem}
There is an integer $\nu=\nu(G)$ such that for every ascending sequence of length $\nu$ of centralizers of elements $g_1,\ldots,g_\nu\in G$ 
$$
 C_G(g_1)\subset C_G(g_2)\subset\cdots\subset C_G(g_\nu),
$$
for some $i<\nu$ we have $C_G(g_i)=C_G(g_{i+1})$.
\end{lem}

\begin{proof}
Considering $G$ as an algebraic subgroup of $\GL_n(\BC)$, the centralizer 
$$
 C_G(g)=\{h:hgh^{-1}=g\}\cap G
$$ 
is defined by $n^2$ quadratic polynomials and the polynomials defining $G$. Since $C_G(g)$ is a group, its irreducible components are its connected components.
Hence, by Bezout's theorem, the number of connected components of $C_G(g)$ is bounded uniformly (independently of $g$). Now if $C_G(g_i)$ and $C_G(g_{i+1})$ have the same dimension and the same number of components, they must coincide, and the lemma is proved since $G$ is finite dimensional.
\end{proof} 

Thus for any semisimple $g\in G$ there is $0\le j< \nu$ such that $C_G(g^{m^j})=C_G(g^{m^{j+1}})$. Denote by $j(g)$ the minimal such $j$.
Set 
$$
\mu=m^\nu.
$$
Let $\gC$ be an irreducible lattice in $G$. Denote by $M=\gC\backslash X$ the corresponding orbifold, and by $\pi: X\to M$ the associated (ramified) covering map. For a subset $Y\subset M$ let $\ti Y=\pi^{-1}(Y)$ be its preimage in $X=\ti M$. 
For $\gc\in\gC\setminus\{ 1\}$ let 
$$
 \overline{\gc}=\gc^{m^{j(\gc)}}~\text{if}~\gc~\text{is hyperbolic, and}~\overline{\gc}=\gc~\text{otherwise}.
$$ 
Obviously $d_{\overline{\gc}}(x)\le\mu d_\gc(x)$.
Since for a hyperbolic element $\Min(g)$ is determined by $C_G(g)$ we have 
\begin{equation}\label{eq:min=min}
\Min(\overline{\gc})=\Min(\overline{\gc}^m),~\text{for every hyperbolic element}~\gc\in\gC\setminus\{1\}.
\end{equation}
Additionally, for $\gc_1,\gc_2$ hyperbolic $[\overline{\gc_1}^m,\overline{\gc_2}^m]=1\implies [\overline{\gc_1},\overline{\gc_2}]=1$.

Note that for every $\gc_0\in\gC$, the set $\{\gc\in\gC:\overline{\gc}=\gc_0\}$ is finite, being compact and discrete. 
Let
$$
 \ti N:=\cup\{ \min(\overline{\gc}):\gc\in\gC\setminus\{1\},~\inf d_{\overline{\gc}}<\gep\}.
$$
Clearly $\ti N$ is a $\gC$-invariant union of totally geodesic submanifolds, and since $\gC$ is discrete, this union is locally finite. 
By Corollary \ref{cor:codim}, unless $G\cong\PSL_2(\BR)$, $\text{codim}_X\ti N\ge 2$, implying that $X\setminus\ti N$ is connected. 
$N=\pi(\ti N)\subset M$ is a singular submanifold consisting of all ramified points and closed geodesics shorter than $\gep$. 

Let $f:\BR^{>0}\to\BR^{\ge 0}$ be a smooth function which tends to $\infty$ at $0$, strictly decreases on $(0,\gep]$ and is identically $0$ on $[\gep,\infty)$. For $\gc\in\gC\setminus\{ 1\}$ and $x\in X\setminus\ti N$ define 
\[
\begin{array}{lll}
\phi_\gc(x)=&f(d_\gc(x))&~\text{if}~\gc~\text{is parabolic}\\
              &f(\mu d_\gc(x))&~\text{if}~\gc~\text{is elliptic}\\
              &f(d_\gc(x)-\inf d_\gc)&~\text{if}~\gc~\text{is hyperbolic and}~\inf d_\gc<\gep\\
              &0&~\text{if}~\gc~\text{is hyperbolic and}~\inf d_\gc\ge\gep,
\end{array}
\]
and set
$$
 \ti\psi(x):=\sum_{\gc\in\gC\setminus\{ 1\}}\phi_{\overline{\gc}}(x).
$$
There are only finitely many nonzero summends for each $x$, since $\phi_{\overline{\gc}}(x)\ne 0\implies d_{\overline{\gc}}(x)\le 2\gep$, and $\gC$ is discrete.
Thus $\ti\psi$ is a well defined smooth function on $X\setminus \ti N$. Note that $\ti\psi(x)$ tends to $\infty$ as $x$ approaches $\ti N$. Since $\ti\psi$ is $\gC$-invariant it induces a function on $M\setminus N$ 
$$
 \psi(x)=\ti\psi(\pi^{-1}(x)).
$$

For $\gd\ge 0$ denote 
$$
 \ti\psi_{\le\gd}=\{ x\in X:\ti\psi(x)\le\gd\}~~\text{and}~~\psi_{\le\gd}=\{ x\in m:\psi(x)\le\gd\}=\pi(\ti\psi_{\le\gd}).
$$
 Then $d_\gc(x)\ge\frac{\gep}{\mu}$ for every $x\in \ti\psi_{\le 0}$ and $\gc\in \gC\setminus\{ 1\}$. Thus, the injectivity radius in $M$ at any point of $\psi_{\le 0}$ is at least $\frac{\gep}{2\mu}$.
 
 
\subsection{The main proposition}

The proof of Theorem \ref{thm:d(Gamma)} relies on the existence of a deformation retract from $M\setminus N$ into an arbitrarily small neighborhood of $\psi_{\le 0}$. The main part consists in showing that the gradient of $\psi$ does not vanish outside $\psi_{\le 0}$ and hence defines a smooth vector field.

\begin{prop}\label{prop:gradient}
For $x\in M\setminus N$, $\nabla \psi(x)=0$ iff $\psi(x)=0$.
\end{prop}

It is more convenient to work in the universal cover where the proposition translates to: for $x\in X\setminus\ti N$, $\nabla \ti\psi(x)=0$ iff $\ti\psi(x)=0$.
We will make use of the following two lemmas:

\begin{lem}\label{lem:ray}
Let $\mathcal{C}\subset X$ be a closed convex set and let $g$ be an isometry of $X$ which preserves $\mathcal{C}$. Let $c:[0,\infty)\to X$ be a 
geodesic ray satisfying 
$$
 c(0)=c([0,\infty))\cap \mathcal{C}=P_\mathcal{C}(c(1)),
$$ 
where $P_\mathcal{C}:X\to\mathcal{C}$ denotes the projection to the nearest point.
Then $h(t):=d_g(c(t))$ is a nondecreasing smooth convex function on $(0,\infty)$. In particular, if $h(t_0)>h(0)$ then $h'(t_0)>0$.
\end{lem}

\begin{lem}\label{lem:abelian-intersect}
$(i)$ The intersection of a finite collection of nonempty sub-level sets corresponding to the displacement functions of commuting isometrics of $X$ is nonempty.

$(ii)$ If $A\leq G$ is an abelian subgroup consisting of semisimple elements. Then
$$
 \cap_{g\in A}\Min (g)\ne\emptyset.
$$
\end{lem}

The proofs of Lemmas \ref{lem:ray} and \ref{lem:abelian-intersect} rely on the following simple:

\begin{clm}\label{clm:1.8}
Let $g$ be an isometry of $X$, let $\mathcal{C}\subset X$ be a closed convex $g$-invariant set and let $x\in X$. Then $d_g(x)\ge d_g(P_\mathcal{C}(x))$.
\end{clm}

\begin{proof}[Proof of \ref{clm:1.8}]
Since $\mathcal{C}$ is $g$-invariant, $P_\mathcal{C}(g\cdot x)=g\cdot P_\mathcal{C}(x)$. Since $P_\mathcal{C}$ is $1$-Lipschitz,
$$
 d_g(x)=d(x,g\cdot x)\ge d(P_\mathcal{C}(x),P_\mathcal{C}(g\cdot x))=d(P_\mathcal{C}(x),g\cdot P_\mathcal{C}(x))=d_g(P_\mathcal{C}(x)).
$$
\end{proof}

\begin{proof}[Proof of Lemma \ref{lem:ray}]
Since $X$ has nonpositive curvature, the distance between the geodesic lines $c(t)$ and $g\cdot c(t)$, $h(t)=d(c(t),g\cdot c(t))$ is convex. By \ref{clm:1.8} $h(t)\ge h(0)$ for any $t>0$. Indeed, 
$$
 h(t)=d_g(c(t))\ge d_g(P_\mathcal{C}(c(t))=h(0).
$$
Thus $h'(0)\ge 0$, and as $h$ is convex, $h'(t)\ge 0$ for every $t\ge 0$. 

If $h(t_0)>0$ for some $t_0$ then, by Lagrange's mean value theorem, $h'(t)>0$ for some $t<t_0$ and since $h'$ is monotonic, $h'(t_0)>0$.
\end{proof}

\begin{proof}[Proof of Lemma \ref{lem:abelian-intersect}]
Note that if $g_1$ and $g_2$ commute, the function $d_{g_2}$, and hence its sublevel sets, are $g_1$-invariant. 
Indeed
$$
 d_{g_2}(g_1\cdot x)=d(g_1\cdot x,g_2g_1\cdot x)=d(g_1\cdot x,g_1g_2\cdot x)=d(x,g_2\cdot x)=d_{g_2}(x).
$$
Suppose that we are given $n$ commuting elements $g_1,\ldots, g_n$ and $n$ non-negative numbers $t_1,\ldots,t_n$ such that $\{ d_{g_i}\le t_i\}$ is nonempty for each $i$. By induction on $n$ we may assume that $\mathcal{A}=\cap_{i\ge 2}\{ d_{g_i}\le t_i\}\ne\emptyset$. Then $\mathcal{A}$ is closed convex and $g_1$-invariant. Hence by \ref{clm:1.8}, $d_{g_1}(P_\mathcal{A}(x))\le d_{g_1}(x)$ for every $x$. By choosing $x$ with $d_{g_1}(x)\le t_1$ we deduce that $\{ d_{g_1}\le t_1\}\cap\mathcal{A}\ne\emptyset$, hence Part $(i)$.

Now by Part $(i)$ it follows that $\mathcal{F}(F)=\cap_{g\in F}\Min (g)\ne\emptyset$ for any finite subset $F\subset A$. Since those $\mathcal{F}(F)$ are complete totally geodesic submanifolds of finite dimension with the finite intersection property, they must all intersect in a common point.
\end{proof}

\begin{proof}[Proof of Proposition \ref{prop:gradient}]
For $x\in X$ let 
$$
 \gS_x=\{\overline{\gc}:\gc\in\gC\setminus\{1\}~\text{and}~\phi_{\overline{\gc}}(x)>0\},~\gC_x=\langle\gS_x\rangle
$$ 
and let $\mathcal{N}_x\lhd\gC_x$ be a normal subgroup of index $i_x\le m_G$ in $\gC_x$ which is contained in a connected nilpotent subgroup of $G$ (see Lemma \ref{lem:margulis}). Let $\gS'_x$ be the finite set consists of all elements in $\mathcal{N}_x$ which are expressible as words of length at most $i_x$ in the alphabet $\gS_x$.

Let $x\in X\setminus (\ti N\cup\ti\psi_{\le 0})$. We are going to show that $\nabla \psi(x)\ne 0$.
We distinguish between three cases:

\medskip

\noindent {\bf Case 1:} Suppose first that $\gC_x$ is finite. Then there is a common fixed point $x_0$ for $\gC_x$. Since $x\notin\ti N$, $x\notin\text{Fix}(\gs)$ for $\gs\in\gS_x$, and we may apply Lemma \ref{lem:ray} with $\mathcal{C}=\{ x_0\}$, $c:[0,\infty)\to X$ being the unit speed geodesic ray from $x_0$ through $x$ and $g=\gs$, and conclude that $\frac{d}{dt}d_\gs(c(t))>0$ for any $t>0$. Since $f'<0$ on $(0,\gep)$ this implies
$$
 \nabla\ti\psi(x)\cdot \dot c(t_0)=\Big(\frac{d}{dt}\Big)_{t_0}\ti\psi(c(t))=\sum_{\gs\in\gS_x}\mu f'(\mu d_\gs(x))\Big(\frac{d}{dt}\Big)_{t_0}d_\gs(c(t))<0,
$$
where $t_0=d(x,x_0)$. In particular $\nabla\ti\psi(x)\ne 0$.

\medskip

Before considering the two other cases note that since $\mathcal{N}_x$ is contained in a connected nilpotent group, by Lie's theorem it is conjugated over $\BC$ to a subgroup of the upper triangular matrices, and hence its commutator consists of unipotent, hence parabolic, elements. In particular:

\begin{itemize}
\item[(a)] If $\gC_x$ consists of semisimple elements then $N_x$ is abelian,

and since for nontrivial parabolic $\gc\in\gC_x$, $\gc^{i_x}$ is a nontrival parabolic in $N_x$: 
\item[(b)] if $\gC_x$ contains a nontrivial parabolic element then $N_x$ contains a nontrivial central parabolic element.
\end{itemize}

\medskip

\noindent {\bf Case 2:} Suppose now that $\gC_x$ is infinite and consists of semisimple elements. 
By (a) $\mathcal{N}_x$ is abelian. 

We claim that the set $\gS_x$ contains a hyperbolic element.
Since $\gC_x$ is infinite, also $\mathcal{N}_x$ is infinite, thus its generating set $\gS'_x$ admits an element of infinite order $\gs_0$. Thus if we suppose that $\gS_x$ consists of torsion elements only, then $\gs_0$, being a word of length at most $m$ in $\gS_x$, would satisfy 
$$
 d_{\gs_0}(x)\le m\max\{d_\gc(x):\gs\in\gS_x\}\le m\frac{\gep}{\mu},
$$ 
but then, as $\overline{\gs_0}=\gs_0^{m^{j(\gs_0)}}$ for $j(\gs_0)\le {\nu-1}$ we would get 
$$
 d_{\overline{\gs_0}}(x)\le m^{\nu-1}d_{\gs_0}(x)< m^\nu\frac{\gep}{\mu}=\gep,
$$
which implies that $\overline{\gs_0}\in\gS_x$, a contradiction. 
Let $\gs_0'$ be an hyperbolic element belonging to $\gS_x$, and note that, by definition, $\gs_0'$ is of the form $\overline{\gc}$ for some $\gc\in\gC\setminus\{ 1\}$.

By Lemma \ref{lem:abelian-intersect}$(ii)$ the closed $\gC_x$-invariant convex set 
$$
 \mathcal{C}=\cap\{\Min(\gc):\gc\in \mathcal{N}_x\}
$$ 
is nonempty. Moreover, since $\gs_0'$ is hyperbolic, by Equation (\ref{eq:min=min}), 
$$
 \Min(\gs_0')=\Min(\gs_0'^m),
$$ 
and since $\gs_0'^m\in \mathcal{N}_x$, we derive that $\mathcal{C}\subset\Min(\gs_0')$. Since $\gs_0'\in \gS_x$ we derive that $\mathcal{C}\subset\ti N$ and since $x\notin\ti N$ we derive that $x\notin\mathcal{C}$. Letting $c:[0,\infty)\to X$ be the constant speed geodesic ray with $c(0)=P_\mathcal{C}(x)$ and $c(1)=x$, we deduce from Lemma \ref{lem:ray} that $\nabla d_\gc\cdot\dot c(1)$ is non-negative for every $\gc\in\gC_x$ while $\nabla d_{\gs_0'}\cdot \dot c(1)$ is positive since 
$$
 c(1)=x\notin\ti N\supset\Min(\gs_0')\supset\mathcal{C}\ni c(0).
$$ 
Thus, as in Case 1, the directional derivative of $\ti\psi$ with respect to $\dot c(1)$ is negative, indicating that $\nabla \ti\psi(x)\ne 0$.

\medskip 

\noindent {\bf Case 3:} Finally, suppose that $\gC_x$ contains a nontrivial parabolic element. 
By (b), $\mathcal{N}_x$ contains a nontrivial central parabolic element $\gc_0$. Since $\gc_0$ is central in $\mathcal{N}_x$ and $\mathcal{N}_x$ is of finite index in $\gC_x$, the conjugacy class $C_{\gC_x}(\gc_0)$ of $\gc_0$ in $\gC_x$ is finite. Pick $\gt$ such that $d_{\gc_0}(x)>\gt>\inf (d_{\gc_0})$ and set 
$$
 \mathcal{C}:=\cap_{\gc\in C(\gc_0)}\{ d_\gc\le \gt\}.
$$
Then $\mathcal{C}$ is nonempty by Lemma \ref{lem:abelian-intersect}$(i)$, and $x\notin\mathcal{C}$ by the choice of $\gt$. Clearly $\mathcal{C}$ is closed convex and $\gC_x$ invariant. Thus, setting $\hat n$ to be the unit tangent at $x$ to the ray $c(t)$ coming from $P_\mathcal{C}(x)$, we derive from Lemma \ref{lem:ray} that $\hat n\cdot\nabla d_\gc(x)\ge 0$ for every $\gc\in\gC_x$. We claim that there must be $\gs\in\gS_x$ for which $\hat n\cdot\nabla d_\gs(x)> 0$. Indeed, assuming this is not the case, one derives that for $\gs\in\gS_x$, $d_\gs (c(t))$ is locally constant, and since $X$ is analytic it is globally constant. Hence $\gc\cdot c(t)$ is parallel to $c(t)$ for every $\gc\in\gS_x$, and since $\gS_x$ generates $\gC_x$ and parallelity is a transitive relation on geodesics in $X$, the same holds for every $\gc\in\gC_x$. This however fails to hold for $\gc=\gc_0$. As in the previous cases, we derive that $\nabla\ti\psi(x)\cdot\hat n<0$, and in particular $\nabla\ti\psi(x)\ne 0$.
\end{proof}


\subsection{Consequences of the main proposition}

\begin{thm}\label{thm:retract}
For every $\gd>0$ there is a deformation retract $r_\gd:M\setminus N\to \psi_{\le\gd}$.
\end{thm}

\begin{proof}
By Proposition \ref{prop:gradient} the smooth function $\psi(x)$ has no positive critical values. 
We claim furthermore that $\psi:M\setminus N\to \BR^{\ge 0}$ is a proper map. Indeed, for $x\in X\setminus\ti N$, $\ti\psi(x)\le a$ implies that $\phi_{\overline{\gc}}(x)\le a$ for every $\gc\in\gC\setminus\{1\}$ which implies that $d_\gc(x)\ge f^{-1}(a)$ if $\gc$ is parabolic, and that $d_\gc(x)\ge \frac{f^{-1}(a)}{\mu}$ if $\gc$ is elliptic or hyperbolic. Thus the injectivity radius at $\pi (x)$ is at least $\frac{f^{-1}(a)}{2\mu}$. It follows that $\psi^{-1}([0,a])$ is contained in the $\big(\frac{f^{-1}(a)}{2\mu}\big)$-thick part of $M$, which is compact since $M$ has finite volume. 

The result follows by standard Morse theory (c.f. \cite{Milnor}, Theorem 3.1).
\end{proof}

An immediate consequence of Theorem \ref{thm:retract} is that $\psi_{\le\gd}$ is nonempty for every $\gd>0$. The proof of \ref{thm:retract} also shows that $\psi_{\le\gd}$ is compact. Since 
$$
 \ti\psi_{\le 0}=\bigcap_{\gd>0}\ti\psi_{\le\gd}.
$$
we derive that the same hold for $\gd=0$:

\begin{cor}
For every $\gd\ge 0$, $\psi_{\le\gd}$ is nonempty and compact.
\end{cor}

Moreover, when $X$ is not $\mathcal{H}^2$ we get:

\begin{cor}
Assume that $\dim X>2$. Then is $\psi_{\le\gd}$ connected for every $\gd\ge 0$.
\end{cor}

\begin{proof}
Since $\gC\le G$ is irreducible and $G$ is center-free, every $\gc\in\gC\setminus\{1\}$ projects nontrivially to every factor of $G$ (cf. \cite[Corollary 5.21]{Raghunathan}).
By Corollary \ref{cor:codim} $\text{codim}_MN\ge 2$ and hence $M\setminus N$ is connected. 
Thus by Theorem \ref{thm:retract} $\psi_{\le\gd}$ is connected for $\gd>0$, and since a descending intersection of compact connected sets is connected, the same holds for $\gd=0$. 
\end{proof}

\begin{rem}\label{rem:K-M}
The nonemptiness of $\psi_{\le 0}$ implies the stronger version of the Kazhdan--Margulis theorem, since the injectivity radius at points of $\psi_{\le 0}$ is at least $\frac{\gep}{2\mu}$; The pre-image in $G$ of the $\frac{\gep}{2\mu}$-neighborhood of $K$ in $G/K$ is an open set which trivially intersect a conjugate of every lattice. Furthermore this set contains a maximal compact subgroup of $G$.
\end{rem}


\subsection{Finishing the proof}\footnote{From here the proof continues as in the special rank one case considered in \cite{BGLS}.}
Assume that $G$ is not locally isomorphic to $\PSL_2(\BR)$. For $G=\PSL_2(\BR)$ the theorem follows from the Gauss--Bonnet formula and the explicate presentation of surface groups.

Since $\ti\psi_{\le\gd}$ is connected nonempty and $\gC$ acts freely on it (as all the fixed points of nontrivial elements are in the singular submanifold $\ti N$), we conclude, from standard covering theory, that:

\begin{cor}
For every $\gd\ge 0$,
$\gC$ is a quotient of $\pi_1(\psi_{\le\gd})$ --- the fundamental group of $\psi_{\le\gd}=\pi (\ti\psi_{\le\gd})=\gC\backslash\ti\psi_{\le\gd}$.
\end{cor}

Instead of rescaling the Haar measure of $G$, let us suppose that it is fixed and corresponds to the Riemannian measure of $X$, and prove that there is a constant $C(G)$ such that $d(\gC)\le C(G)\vol(G/\gC)$.
Recall from the proof of Theorem \ref{thm:retract} that $\psi_{\le 0}$ is contained in the $\alpha$-thick part of $M$ for $\ga=\frac{\gep}{2\mu}$. Let $\mathcal{S}$ be a maximal $\ga$-discrete subset of $\psi_{\le 0}$. For $t>0$ denote by $v (t)$ the volume of a $t$-ball in $X$.
Since the $\ga/2$ balls centered at points of $\mathcal{S}$ are pairwise disjoint and isometric to a $\ga/2$ ball in $X$, the size of $\mathcal{S}$ is bounded by $\vol (M)/v(\frac{\ga}{2})$. Moreover, the union of the $\ga$ balls centered at points of $\mathcal{S}$ covers $\psi_{\le 0}$. Denote this union by $U$. Choose $\gd_0>0$ sufficiently small so that $\psi_{\le\gd_0}$ is contained in $U$.

\begin{lem}
$\pi_1(\psi_{\le\gd_0})$ is a quotient of $\pi_1(U)$.
\end{lem}

\begin{proof}
The inclusion $i:\psi_{\le\gd_0}\to U$ induces a map $i_*:\pi_1(\psi_{\le\gd_0})\to\pi_1(U)$, and the deformation retract $r=r_{\gd_0}$ restricted to $U$, $r:U\to \psi_{\le\gd_0}$ induces a map $r_*:\pi_1(U)\to\pi_1(\psi_{\le\gd_0})$. Since $r\circ i$ is the identity on $\psi_{\le\gd_0}$, we see that $r_*\circ i_*$ is the identity on $\pi_1(\psi_{\le\gd_0})$. It follows that $r_*:\pi_1(U)\to\pi_1(\psi_{\le\gd_0})$ is onto.
\end{proof}

Since the sectional curvature is nonpositive on $M$ the $\ga$-balls centered at points of $\mathcal{S}$ are convex, and hence any nonempty intersection of such is convex, hence contractible. Thus these balls form a good cover of $U$, in the sense of \cite{BoTu}, and the nerve $\mathfrak{N}$ of this cover is homotopic to $U$. Now $\pi_1(U)\cong\pi_1(\mathfrak{N})$ has a generating set of size $E(\mathfrak{N})$ --- the number of edges of the $1$-skeleton $\mathfrak{N}^1$. To see this one may choose a spanning tree $\mathcal{T}$ to the graph $\mathfrak{N}^1$ and pick one generator for each edge belonging to $\mathfrak{N}^1\setminus\mathcal{T}$. 
Finally note that the edges of $\mathfrak{N}$ correspond to pairs of points in $\mathcal{S}$ which are of distance at most $2\ga$. Thus the degree of each vertex of $\mathfrak{N}$ is at most $v(2.5\ga)/v(\ga/2)$. Thus
$$
 d(\gC)\le |E(\mathfrak{N})|\le\frac{\vol(M)\cdot v(2.5\ga)}{2v(\ga/2)^2},
$$
and we can take $C(G)=\frac{v(2.5\ga)}{2v(\ga/2)^2}$.\qed

\begin{rem}
The same proof with minor changes applies in the more general context of analytic Hadamard spaces with bounded curvature, as long as the analog of Lemma \ref{lem:codim} is satisfied: {\it For every $d$ there is a constant $C$ such that for any complete $d$-dimensional analytic Riemannian manifold $M$ with sectional curvature bounded between $0$ and $-1$, we have}
$$
 d(\pi_1(M))\le C\vol (M),
$$
provided the universal cover $\ti M$ does not admit a submanifold of co-dimension $\le 1$ with a nontrivial Euclidean de-Rahm factor
\end{rem}

\medskip

\noindent {\bf Acknowledgment:} {\it The author acknowledges the financial support from the European
Research Council (ERC)/ grant agreement 203418.}\\


\noindent {\bf Email:} gelander@math.huji.ac.il

\medskip

\noindent {\bf Address:} Einstein Institute of Mathematics, 
The Hebrew University,
Jerusalem, 91904, Israel


\begin{thebibliography}{999999}


\bibitem[BGS]{BGS} W. Ballmann, M. Gromov, V. Schroeder, 
{\it Manifolds of Nonpositive Curvature}, Birkhauser, 1985.

\bibitem[Be]{Be} M. Belolipetsky, Counting maximal arithmetic subgroups. With an appendix by Jordan Ellenberg and Akshay Venkatesh.  Duke Math. J.  140  (2007),  no. 1, 1--33.

\bibitem[BGLS]{BGLS} M. Belolipetsky, T. Gelander, A. Lubotzky, A. Shalev, Counting Arithmetic Lattices and Surfaces, to appear in Ann. of Math.

\bibitem[BL]{BL} M. Belolipetsky, A. Lubotzky, Counting manifolds and class field towers, preprint.



\bibitem[BT]{BoTu} R. Bott, L.W. Tu, {\it Differential Forms in Algebraic 
Topology}, Springer-Verlag, 1982.

\bibitem[BH]{BH} R. Bridson, A. Haefliger, {\it Metric Spaces of Non-Positive 
Curvature}, Springer, 1999.

\bibitem[BGLM]{BGLM} M. Burger, T. Gelander, A. Lubotzky, S. Mozes, Counting hyperbolic manifolds, Geom. Funct. Anal. 12 (2002), no. 6, 1161--1173. 




\bibitem[FGT]{Fi-Gr} T. Finis, F. Grunewald, P. Tirao, The cohomology of lattices in $\SL(2,\BC)$, {\em preprint}.

\bibitem[GR]{Ga-Ra} H. Garland, M.S. Raghunathan, Fundamental domains for lattices in (R-)rank 1 Lie groups, Ann. of Math. {\bf 92} (1970), 279-326.


\bibitem[Ge]{Ge} T. Gelander, Homotopy type and volume of locally symmetric manifolds.  Duke Math. J.  124  (2004),  no. 3, 459--515.




\bibitem[Lu1]{Lub1} A. Lubotzky, Free quotients and the first Betti number of
some hyperbolic manifolds. Transformation Groups 1 (1996), 71--82

\bibitem[Lu2]{Lub2} A. Lubotzky, \textit{On finite index subgroups of linear groups}
Bull. London Math. Soc. 19 (1987), no. 4, 325--328.




\bibitem[K]{Kazhdan} D.A. Kazhdan, Connection of the dual space of a group with the structure of its closed subgroups, Functional Analysis and Application {\bf 1}
(1967), 63-65. 

\bibitem[KM]{KM} D.A. Kazhdan, G.A. Margulis, A proof of Selberg's hypothesis. Math Sbornilr (N.S.) 75 (117), 162-168 (1968) [Russian].

\bibitem[LS]{LS} A. Lubotzky and D. Segal, {\it Subgroup Growth}, Birkhauser,
2003.

\bibitem[M]{Milnor} J. Milnor, {\it Morse theory}, Based on lecture notes by M. Spivak and R. Wells. Annals of Mathematics Studies, No. 51 Princeton University Press, Princeton, N.J. 1963 vi+153 pp.

\bibitem[R]{Raghunathan} M.S. Raghunathan, {\it Discrete Subgroups of Lie Groups}, 
Springer, New York, 1972.

\bibitem[SV]{SV} R. Sharmaandt and N. Venkataramana, Generators for arithmetic groups, preprint.

\bibitem[Th]{Th} W.P. Thurston, {\it Three-Dimensional Geometry and Topology},
Volume 1, Princeton univ. press, 1997.


\end{thebibliography}
\end{document}